\documentclass[english]{amsart}
\usepackage{amsfonts, amsmath, amssymb, amscd, amsthm, array, graphicx}
\usepackage[latin9]{inputenc}
\usepackage{amsmath}
\usepackage{stmaryrd}
\usepackage{graphicx}
\usepackage{amssymb}
\usepackage{dsfont}
\usepackage{babel}
\usepackage{color}
\usepackage[all]{xy}

\usepackage{mathtools}
\usepackage{hyperref,url}

\makeatletter
\def\blfootnote{\xdef\@thefnmark{}\@footnotetext}
\makeatother
\usepackage{tikz}
\usetikzlibrary{arrows,shapes,snakes,automata,backgrounds,petri}
\usetikzlibrary{automata} 
\usetikzlibrary[automata] 
\usetikzlibrary{matrix,decorations.pathreplacing}
\usepackage{geometry}
\usepackage{pdflscape}
\usepackage{graphicx}
\usetikzlibrary{external,automata,trees,positioning,shadows,arrows,shapes.geometric}

\newlength\Textht
\newtheorem{thm}{Theorem}[section]
\newtheorem*{thm*}{Theorem}
\newtheorem{cor}[thm]{Corollary}
\newtheorem{lem}[thm]{Lemma}

\theoremstyle{definition}
\newtheorem{dfn}[thm]{Definition}

\newtheorem{rem}[thm]{Remark}

\newcommand\ZZ{\mathbb{Z}}

\newcommand{\N}{{\mathbb N}}

\newcommand{\GL}{\operatorname{GL}}

\newcommand{\Aut}{\operatorname{Aut}}

\newcommand{\End}{\operatorname{End}}

\newcommand{\BS}{\operatorname{BS}}
\newcommand{\Fix}{\operatorname{Fix}}

\newcommand{\stab}{\operatorname{Stab}}

\newcommand{\lcm}{\operatorname{lcm}}

\thanks{ }

\title{On fixed points and stabilizers in solvable Baumslag--Solitar groups}

\author{Oorna Mitra}
\address{Indian Institute of Science Education and Research (IISER) Kolkata, Mohanpur, Nadia - 741246, West Bengal, India.}
\email{urna.mitra@gmail.com}

\author{Ramya Nair}
\address{The Institute of Mathematical Sciences, CIT Campus, Tharamani, Chennai-600113, Tamil Nadu, India. }
\email{ramyanair59@gmail.com}
\begin{document}

\begin{abstract}
    In this article, we study the fixed-point subgroups of the solvable Baumslag-Solitar groups $\BS(1,n)= \langle a, t \mid t a t^{-1} = a^{n} \rangle$, $n>1$ of automorphisms and endomorphisms. We also investigate the stabilizers of subgroups of $\BS(1,n)$, considered as subgroups of the group of automorphisms  and submonoids of the monoid of endomorphisms of $\BS(1,n)$. We show that the fixed-point subgroups of automorphisms are either infinite cyclic (in which case, a generator is computable), or they are equal to $\mathbb{Z}\left[\tfrac{1}{n}\right]$, an infinitely generated abelian group. We further prove that the stabilizer subgroup of an element in $\BS(1,n)$ is either a finitely generated abelian group whose rank equals the number of distinct prime divisors of $n$ (and in this case, a finite generating set is computable), or it is $\mathbb{Z}\left[\tfrac{1}{n}\right]$. As a corollary, we show that for all $k \in \mathbb{N}$, every element of $\BS(1,n)$ has a unique $k$-th root. We then proceed to examine the behaviour of fixed-point subgroups and stabilizers under endomorphisms and find similar results. We prove that the fixed point subgroups of endomorphisms are again infinite cyclic or $\mathbb{Z}\left[\tfrac{1}{n}\right]$, but the stabilizer submonoids are always infinitely generated.
\end{abstract}
\maketitle
\section{Introduction}

Given a group $G$ and an automorphism (or endomorphism) $\varphi$ of $G$, the fixed-point subgroup is defined by $\Fix(\varphi)=\{g \mid \varphi(g)=g\},$ which is a subgroup of $G$. The study of fixed points became prominent for the first time in the context of free groups. In 1975, Dyer and Scott \cite{DS} formulated what became known as the Scott Conjecture, which states that the rank of the fixed-point subgroup of an automorphism of a finitely generated free group is bounded above by the rank of the group itself. This conjecture was later proved by Bestvina and Handel \cite{BH} in 1992. Since then, fixed-point subgroups have been investigated in a variety of settings, particularly in groups with negative curvature, including hyperbolic \cite{WN} and relatively hyperbolic groups \cite{AD}. More recently, fixed-point subgroups have also been studied in groups without negative curvature, such as Artin groups \cite{Oli}. Various properties of fixed subgroups of generalized Baumslag-Solitar groups have been investigated by Jones and Logan \cite{JL}. For results on fixed subgroups in other classes of groups, see \cite{JWWZ}, \cite{LZ}, \cite{RV}, \cite{Syk}, \cite{WZ}.

Given a group $G$, $\Aut(G)$ denotes the group of automorphisms of $G$ and $\End(G)$ denotes the monoid of endomorphisms of $G.$
Given an element $g \in G,$ we define the stabilizer $\stab(g)=\{\varphi \in \Aut(G) \mid \varphi(g)=g\}$ of $g$, which turns out to be a subgroup of $\Aut(G).$ We also define the endo-stabilizer $e\text{-}\stab(g)=\{\varphi \in \End(G) \mid \varphi(g)=g\}$ of $g$, which is a submonoid of $\End(G).$

In this paper, we describe the fixed subgroups of automorphisms and endomorphisms in $\BS(1,n).$ We also describe the stabilizers of elements of $\BS(1,n)$, considered as subgroups of the group of automorphisms  and submonoids of the monoid of endomorphisms of $\BS(1,n)$. We heavily rely on the description of automorphisms and endomorphisms of $\BS(1,n), n >2$ given by O'Niell \cite{ON}. The techniques involved in our proofs are algebraic and combinatorial in nature.

The paper is organized as follows. 

In Section~\S2, we begin by recalling several equivalent descriptions of the groups $\BS(1,n)$ and describe the automorphisms and endomorphisms of $\BS(1,n)$, as given by O'Niell \cite{ON}. We also present various descriptions of the automorphism group $\Aut(\BS(1,n))$, following \cite{MRV}. 

In Section~\S3, we study $\Fix(\varphi),$ for $\varphi \in \Aut((\BS(1,n))$ and $\stab(g),$ for $g \in \BS(1,n).$ The main theorems of this section are as follows.
\begin{thm} 
    Let $\varphi\in \Aut(\BS(1,n))$ be a non-trivial automorphism, then either $\Fix\varphi\cong\ZZ$ or $\Fix\varphi\cong \ZZ \left[\tfrac{1}{n}\right]$. When $\Fix\varphi\cong\ZZ$, its generator is computable.
\end{thm}

\begin{thm} 
    Let $g$ be a non-trivial element of $\BS(1,n)$. Then  either $\stab(g) \cong \mathbb{Z}\left[\frac{1}{n}\right]$, or $\stab(g)$ is a finitely generated abelian group with rank equals to the number of distinct primes present in $n$ and a finite set of generators is computable.
\end{thm}
In this section, we also describe the auto-fixed closure of $\BS(1,n)$, which is a notion introduced by Martino and Ventura in \cite{MV}. We observe that $\stab(g)=\stab(g^k),$ for all $k \in \mathbb{Z}$. This allows us to conclude that for all $k \in \mathbb{N}$, every element of $\BS(1,n)$ has a unique $k$-th root.

In Section~\S4, we investigate the fixed-point subgroups $\Fix(\varphi)$ for endomorphisms $\varphi \in \End(\BS(1,n))$ and the stabilizer submonoids $e\text{-}\stab(g)$ for elements $g \in \BS(1,n)$. The principal theorems of this section are presented here.
\begin{thm}\label{e-fix}
   Let $\varphi\in \End(\BS(1,n))$ be a non-trivial endomorphism, then either $\Fix\varphi\cong\ZZ$ or $\Fix\varphi\cong \ZZ \left[\tfrac{1}{n}\right]$ or $\Fix(\varphi)$ is trivial. When $\Fix\varphi\cong\ZZ$, its generator is computable.
\end{thm}

\begin{thm}
    Let $g \in \BS(1,n)$ be a non-trivial element. Then $e\text{-}\stab(g)$ is an infinitely generated submonoid of $\End(\BS(1,n))$.
\end{thm}

We describe the endo-fixed closure of $\BS(1,n)$ in this section. We also prove that for any $\varphi \in \End(\BS(1,n)), g \in \BS(1,n), \Fix(\varphi)=\Fix(\varphi^k)$ for all $k \in \mathbb{N}$ and $e\text{-}\stab(g)=e\text{-}\stab(g^k)$ for all $k\in \ZZ$
\section{Preliminaries}
 
\noindent The \textbf{Baumslag--Solitar groups} \( \BS(m,n) \), with $m, n \in \mathbb{Z} \setminus \{0\}$, are defined by the presentation
\[
\BS(m,n) = \langle a, t \mid t a^{m} t^{-1} = a^{n} \rangle.
\]
These groups have served as a testing ground for many new ideas in combinatorial and geometric group theory, see \cite{BS} \cite{FM} \cite{FM1} for example.
Among them, the only solvable groups are of the form \( \BS(1,n) \); when \( n \in \mathbb{Z} \setminus \{0\} \). Below we will give other descriptions of the group $\BS(1,n)$. 

These groups are HNN-extension of $\mathbb{Z} =$ \(\langle a \rangle\) relative to the subgroups \(\mathbb{Z}\) and \(n\mathbb{Z}\), and the isomorphism 
\[
\varphi : \mathbb{Z} \to n\mathbb{Z}, \quad \varphi(a) = a^{n}.
\]

There is a well defined onto homomorphism $\BS(1,n) \to \mathbb{Z}$, given by $a \mapsto 0$ and $t \mapsto 1$, the kernel of which turns out to be the normal closure $\langle\langle a\rangle\rangle$ generated by $a$, isomorphic to $\mathbb{Z
}[\frac{1}{n}].$ Thus $\BS(1,n)$ fits into the short exact sequence $\begin{array}{ccccccccc} 1 & \to\, & \mathbb{Z}\left[\frac{1}{n}\right] & \, \to & \BS(1,n) & {\to} & \ZZ  & \to & 1.\end{array}$
that splits. This shows that $\BS(1,n)$ is a metabelian group of the form \[
\BS(1, n) \cong \mathbb{Z}\!\left[\tfrac{1}{n}\right] \rtimes \mathbb{Z}.
\]
Throughout the paper, we will work with the semidirect product structure of $\BS(1,n)$, the elements of $\BS(1,n)$ will be denoted by $(\gamma, c)$, where $\gamma \in \mathbb{Z}\left[\frac{1}{n}\right]$ and $c \in \mathbb{Z},$ and the operation is given by \begin{equation}\label{eq:1}
(\gamma_1,c_1).(\gamma_2,c_2) := (\gamma_1n^{c_2}+\gamma_2, c_1+c_2).
\end{equation}

\noindent For any $n \in \mathbb{Z}$ with $n \neq \pm 1$, consider the monomorphism $\BS(1,n) \longrightarrow \GL_2\!\left(\mathbb{Z}\!\left[\tfrac{1}{n}\right]\right)$ given by $(\gamma,c) \longmapsto
\begin{pmatrix}
n^{c} & 0 \\
\gamma & 1
\end{pmatrix}.$ Using this, we see that the group $\BS(1,n)$, for $n \neq \pm 1$, can be viewed as a linear group. The above monomorphism can also be used to show that $\BS(1,n) \cong \BS(1,-n)\quad \text{whenever } n \neq \pm 1.$ Henceforth we work with $\BS(1,n),$ where $n\geq2$.

The following lemma will be used throughout the paper.

\begin{lem}\label{basic}
    For $(\gamma, c) \in \BS(1,n),$ where $\gamma \in \mathbb{Z}\left[\frac{1}{n}\right]$ and $c \in \mathbb{Z},$ the following statements hold true: 
    \begin{enumerate}
\item $(\gamma, c)^{-1} = (-\gamma n^{-c},-c);$
\item For $k \in \mathbb{Z}, (\gamma, c)^k=\left(\frac{n^{ck}-1}{n^c-1}\gamma,ck\right).$
\end{enumerate}
\end{lem}

\begin{proof}
    For (i), using the composition rule in ~\ref{eq:1}, an easy calculation shows that $(\gamma, c).(-\gamma n^{-c},-c)=(-\gamma n^{-c},-c).(\gamma,c)=(0,0),$ which is the identity element of $\BS(1,n).$ Hence, $(\gamma, c)^{-1} = (-\gamma n^{-c},-c).$ 

    For (ii), we first consider the case when $k \geq 0.$ The statement is obviously true for $k=0,k=1.$ Assume that $(\gamma, c)^k=\left(\frac{n^{ck}-1}{n^c-1}\gamma,ck\right).$ Then using the composition rule in ~\ref{eq:1}, $(\gamma, c)^{k+1}=\left(\frac{n^{ck}-1}{n^c-1}\gamma,ck\right).(\gamma,c)=\left(\frac{n^{ck}-1}{n^c-1}\gamma.n^c+ \gamma, (k+1)c\right)=\left(\gamma \left(\frac{n^{c(k+1)}-n^c}{n^c-1} + 1\right), (k+1)c\right)=\left(\frac{n^{c(k+1)}-1}{n^c-1}\gamma,c(k+1)\right).$ So we have, $(\gamma, c)^k=\left(\frac{n^{ck}-1}{n^c-1}\gamma,ck\right)$ for all $k \geq 0.$
    
    For $k<0,$ we apply the above for $-k \geq 0$ and get $(\gamma, c)^k=((\gamma, c)^{-1})^{-k}=(-\gamma n^{-c},-c)^{-k}=\left(\frac{n^{ck}-1}{n^{-c}-1}. -\gamma n^{-c},ck\right)= \left(\frac{n^{ck}-1}{n^c-1}\gamma,ck\right).$ \newline
    This completes the proof.
\end{proof}

\noindent The following description of finitely generated subgroups of $\BS(1,n)$ will be used multiple times throughout the paper:
\begin{lem} \label{subgroups}
Let $H$ be a finitely generated subgroup of $\BS(1,n)$. Then $H$ is infinite cyclic or of finite index.
\end{lem}

\begin{proof}
Let $H=\langle (\gamma_1,k_1),\ldots,(\gamma_l,k_l)\rangle$. Consider the epimorphism
$\psi: BS(1,n)\to\mathbb{Z}$, where $a\mapsto 0,\; t\mapsto 1$. Its image $\psi(H)$ is cyclic, say generated by $k$. If $k=0$, then $H$ is a finitely generated subgroup of $\ker \psi\cong \mathbb{Z}\!\left[\tfrac{1}{n}\right]$ and is therefore infinite cyclic. For $k\neq0$, we can write $H=\langle (\gamma,k),(\gamma_1',0),\ldots,(\gamma_l',0)\rangle$, for some $\gamma\in\mathbb{Z}\!\left[\tfrac{1}{n}\right]$ and $(\gamma_i',0)\in\ker\psi$ for all $1\leq i\leq l$. Since $\ker\psi\cong\mathbb{Z}\!\left[\tfrac{1}{n}\right]$ and every finitely generated subgroup of $\mathbb{Z}\!\left[\tfrac{1}{n}\right]$ is cyclic, we obtain $H=\langle (\gamma,k),(\gamma'',0)\rangle \quad\text{for some }(\gamma'',0)\in\ker\psi.$ If $\gamma''=0$, then $H$ is infinite cyclic. If $\gamma''\neq 0$, then by the discussion in section 3 of \cite{Bo}, the subgroup $H$ has finite index in $\BS(1,n)$ and $H\cong \BS(1,n^{k})$.

\end{proof}

Now we are going to describe the automorphisms of $\BS(1,n)$ and structure of the group $\Aut(\BS(1,n))$ given by O'Niell \cite{ON} and Mitra, Roy, Ventura \cite{MRV}.
For $n\neq \pm 1$, the automorphism group of $\BS(1,n)$ is 
 $$
\Aut(\BS(1,n))=\left\{\varphi_{\alpha, \beta} \mid \alpha\in \ZZ[\tfrac{1}{n}]^*,\, \beta \in \ZZ[\tfrac{1}{n}]\right\},
 $$
where $\varphi_{\alpha, \beta}(0,1)=(\beta,1),$ and $\varphi_{\alpha, \beta}(1,0)=(\alpha,0).$ Here, $\ZZ[\tfrac{1}{n}]^*$ denotes the group of units in $\ZZ[\tfrac{1}{n}]$, which is isomorphic to $\mathbb{Z}/2\mathbb{Z} \times \mathbb{Z}^r$, where $r$ is the number of distinct primes present in $n$.

For an arbitrary element $(\gamma, c)$, where $\gamma = \frac{l}{n^k}, \qquad l, k \in \mathbb{Z}, \; k > 0$, we have $(\gamma, c) = (0,1)^{c} \cdot (0,1)^{k} \cdot (1,0)^{l} \cdot (0,1)^{-k}$. Hence, $\varphi_{\alpha,\beta}(\gamma, c)
= \varphi_{\alpha,\beta}\big((0,1)^{c} \cdot (0,1)^{k} \cdot (1,0)^{l} \cdot (0,1)^{-k}\big)
= (\beta,1)^{c} \cdot (\beta,1)^{k} \cdot (\alpha,0)^{l} \cdot (\beta,1)^{-k}$. Note that $(\beta,1)^m = \big(\beta(1+n+\cdots+n^{m-1}), m\big)
= \left(\beta \frac{n^m - 1}{n - 1}, m\right),
\forall\, m \in \mathbb{Z}\setminus {\{0\}}$, and $(\alpha,0)^l = (l\alpha, 0),  \forall\, l \in \mathbb{Z}\setminus {\{0\}}$. 

Therefore,
\[
\begin{aligned}
\varphi_{\alpha,\beta}(\gamma, c)
&= \left(\beta \frac{n^c - 1}{n - 1}, c\right)
   \cdot \left(\beta \frac{n^k - 1}{n - 1}, k\right)
   \cdot (l\alpha, 0)
   \cdot \left(\beta \frac{n^{-k} - 1}{n - 1}, -k\right) \\
&= \left(\beta \frac{n^c - 1}{n - 1}, c\right)
   \cdot \left(l\alpha + \beta \frac{n^k - 1}{n - 1}, k\right)
   \cdot \left(\beta \frac{n^{-k} - 1}{n - 1}, -k\right) \\
&= \left(\beta \frac{n^c - 1}{n - 1}, c\right)
   \cdot \left(\alpha \frac{l}{n^k}, 0\right) \\
&= \left(\alpha \frac{l}{n^k} + \beta \frac{n^c - 1}{n - 1}, c\right) \\
&= \left(\alpha \gamma + \beta \frac{n^c - 1}{n - 1},\, c\right).
\end{aligned}
\]

Now if $\varphi_{\alpha, \beta}, \varphi_{\alpha', \beta'} \in \Aut(\BS(1,n))$, then $\varphi_{\alpha, \beta} \circ \varphi_{\alpha', \beta'} (\gamma,c) = \varphi_{\alpha, \beta}\left(\alpha' \gamma + \beta' \frac{n^c - 1}{n - 1}, c\right)=\left(\alpha (\alpha' \gamma + \beta' \frac{n^c - 1}{n - 1})+\beta \frac{n^c - 1}{n - 1}, c\right)=\left(\alpha \alpha' \gamma + (\alpha \beta' + \beta) \frac{n^c - 1}{n - 1}, c\right)=\varphi_{\alpha \alpha', \alpha \beta' +\beta}(\gamma,c).$ So we have, $\varphi_{\alpha, \beta} \circ \varphi_{\alpha', \beta'}= \varphi_{\alpha \alpha', \alpha \beta' +\beta}$.

From the above composition rule, it is clear that there is a well defined onto homomorphism $\Aut(\BS(1,n)) \to \ZZ[\tfrac{1}{n}]^*$ given by $\varphi_{\alpha, \beta} \mapsto \alpha$, whose kernel is $\left\{\varphi_{1, \beta} \middle| \beta \in \mathbb{Z}\left[\frac{1}{n}\right]\right\} \cong \mathbb{Z}\left[\frac{1}{n}\right]$. Thus $\Aut(\BS(1,n))$ fits into the short exact sequence $\begin{array}{ccccccccc} 1 & \to\, & \mathbb{Z}\left[\frac{1}{n}\right] & \, \to & \Aut(\BS(1,n)) & {\to} & \ZZ[\tfrac{1}{n}]^*  & \to & 1.\end{array}$ that splits. This shows that $\Aut(\BS(1,n))$ is also a metabelian group of the form \[
\Aut(\BS(1, n)) \cong \mathbb{Z}\!\left[\tfrac{1}{n}\right] \rtimes \ZZ\left[\tfrac{1}{n}\right]^*.
\]
The composition rule of the group $\Aut(\BS(1,n))$ ensures that this group can also be viewed as a linear group, where the monomorphism $\Aut(\BS(1,n)) \to \GL_2(\mathbb{Z}\left[\frac{1}{n}\right])$ is given by $\varphi_{\alpha, \beta} \mapsto \begin{pmatrix} \alpha & \beta \\ 0 & 1 \end{pmatrix}.$ 
Throughout this paper, we will use the following description of $\Aut(\BS(1,n)).$
\[
\Aut(\BS(1, n)) = \left\{\begin{pmatrix} \alpha & \beta \\ 0 & 1 \end{pmatrix} \middle| \alpha\in \ZZ\left[\tfrac{1}{n}\right]^*,\, \beta \in \ZZ\left[\tfrac{1}{n}\right]\right\}
\] where the composition rule is given by usual matrix multiplication and the action of $\Aut(\BS(1,n))$ on $\BS(1,n)$ is given by 
\[
\begin{pmatrix} \alpha & \beta \\ 0 & 1 \end{pmatrix} (\gamma,c) = \left(\alpha \gamma + \beta \frac{n^c - 1}{n - 1}, c\right)
\]
For detailed calculation, see subsection~\S 2.2 titled \textit{Automorphisms of $\BS(1,n)$} in \cite{MRV}.

Now we are going to describe the endomorphisms of $\BS(1,n).$ Following Proposition 2.1 in \cite{ON}, there are two types of endomorphisms. 

\noindent \textit{Type (i) :} We denote the set of all \textit{Type (i)} endomorphisms as  \[
\End^{\mathrm{I}}(\BS(1,n)) := \left\{\begin{pmatrix} \alpha & \beta \\ 0 & 1 \end{pmatrix} \middle| \alpha\in \ZZ\left[\tfrac{1}{n}\right] \setminus{\{0\}},\, \beta \in \ZZ\left[\tfrac{1}{n}\right]\right\}
\] where the composition rule is given by usual matrix multiplication and the action of $\End^{\mathrm{I}}(\BS(1,n))$ on $\BS(1,n)$ is given by 
\[
\begin{pmatrix} \alpha & \beta \\ 0 & 1 \end{pmatrix} (\gamma,c) = \left(\alpha \gamma + \beta \frac{n^c - 1}{n - 1}, c\right)
\] 

\noindent \textit{Type (ii) :} We denote the set of all \textit{Type (ii)} endomorphisms as  \[
\End^{\mathrm{II}}(\BS(1,n)) := \left\{ \varphi_{(\gamma,c)} \mid (\gamma,c) \in \BS(1,n)\right\},
\] where $\varphi_{(\gamma,c)}(1,0)=(0,0), \varphi_{(\gamma,c)}(0,1)=(\gamma,c).$ \newline
For $(\gamma,c),(\gamma',c') \in \BS(1,n), \varphi_{(\gamma,c)} \circ \varphi_{(\gamma',c')}(1,0)=(0,0), \varphi_{(\gamma,c)} \circ \varphi_{(\gamma',c')}(0,1)= \varphi_{(\gamma,c)}(\gamma',c').$
Now, let $\gamma' = \frac{l}{n^k},  l, k \in \mathbb{Z}, \; k > 0$, we have $(\gamma', c') = (0,1)^{c'} \cdot (0,1)^{k} \cdot (1,0)^{l} \cdot (0,1)^{-k}$. Hence, $\varphi_{(\gamma, c)}(\gamma',c')
= \varphi_{(\gamma,c)}\big((0,1)^{c'} \cdot (0,1)^{k} \cdot (1,0)^{l} \cdot (0,1)^{-k}\big)
= (\gamma,c)^{c'}=\left(\frac{n^{cc'}-1}{n^c-1}\gamma,cc'\right)$ by using Lemma \ref{basic}. \\
So we have the composition rule for \textit{Type (ii)} endomorphisms as follows
\[ \varphi_{(\gamma,c)} \circ \varphi_{(\gamma',c')} = \varphi_{\left(\frac{n^{cc'}-1}{n^c-1}\gamma,cc'\right)}\]

\begin{rem}
    $\End^{\mathrm{I}}(\BS(1,n))$ and $\End^{\mathrm{II}}(\BS(1,n))$ are two disjoint subsets of the monoid $\End(\BS(1,n))$ that are closed under composition.
\end{rem}
\section{Fixed subgroups of automorphisms and stabilizers}
\begin{dfn}
Let $G$ be a group and $\Aut(G)$ be its automorphism group. Let $\varphi\in \Aut(G)$, then the fixed subgroup of $\varphi$ denoted by $\Fix\varphi$ is given by $\Fix\varphi=\left\{g\in G\middle| \varphi(g)=g\right\}$  
\end{dfn}

Clearly, $\Fix\varphi$ forms a group. 

\begin{thm} \label{fixed}
    Let $\varphi\in \Aut(\BS(1,n))$ be a non-trivial automorphism, then either $\Fix\varphi\cong\ZZ$ or $\Fix\varphi\cong \ZZ \left[\tfrac{1}{n}\right]$. When $\Fix\varphi\cong\ZZ$, its generator is computable.
\end{thm}
\begin{proof}
    Let $\varphi=\begin{pmatrix} \alpha & \beta \\ 0 & 1 \end{pmatrix} $ for some $\alpha\in \ZZ\left[\tfrac{1}{n}\right]^*$ and $\beta \in \ZZ\left[\tfrac{1}{n}\right].$ If $(\gamma,c)\in \Fix\varphi$, then $\varphi(\gamma,c)=(\gamma,c)$, that is, $\begin{pmatrix} \alpha & \beta \\ 0 & 1 \end{pmatrix}(\gamma,c)=(\gamma,c)$ which gives us $\left(\alpha \gamma + \beta \frac{n^c - 1}{n - 1}, c\right)=(\gamma,c)$ and hence 
    \begin{equation}\label{eq:2}
        (\alpha-1) \gamma + \beta \frac{n^c - 1}{n - 1}=0.
    \end{equation}    
    Now we will consider two cases. \\
    \textit{Case (i): $\alpha=1$}
    
    
    \noindent If $\alpha=1$, then $\beta\neq 0$ as $\varphi$ is non-trivial and $\beta \frac{n^c - 1}{n - 1}=0$. This implies $\frac{n^c - 1}{n - 1}=0\implies c=0$. Therefore, for $\varphi=\begin{pmatrix} 1 & \beta \\ 0 & 1 \end{pmatrix}$, $\Fix\varphi=\left\{(\gamma,0)\middle|\gamma\in \ZZ\left[\tfrac{1}{n}\right]\right\}\cong \ZZ\left[\tfrac{1}{n}\right]$. 
    
    \noindent \textit{Case (ii) : $\alpha \neq 1$} 
    
    \noindent For $\alpha\neq 1$, by equation \ref{eq:2} the map $\BS(1,n)\longrightarrow\ZZ$ where $(\gamma,c)\longmapsto c$ when restricted to $\Fix\varphi$, is injective and hence $\Fix\varphi\leq \ZZ$. For $(\gamma,c)$ to be in $\Fix\varphi$, $\gamma=-\frac{\beta}{\alpha-1} \frac{n^c - 1}{n - 1}$ (by equation ~\ref{eq:2}). Let $n=p_1^{k_1}p_2^{k_2}..p_r^{k_r}$ where $k_i\in \ZZ_{>0}$, then as $\alpha\in \ZZ \left[\tfrac{1}{n}\right]^* $, $\alpha=\prod_{j=1}^r p_j^{s_j}$ where $s_j\in \ZZ \;\forall j$. Let $J_-:=\{\,j:\, s_j<0\,\},J_+:=\{\,j:\, s_j>0\,\} $ and $J_0:=\{\,j:\, s_j=0\,\}$. Substituting $\alpha=\prod_{j=1}^r p_j^{s_j}$ into $\gamma=\frac{\beta}{\alpha-1}\,\frac{n^c-1}{n-1}$ gives $$\gamma =\frac{\beta\,(n^c-1)}
{\big(\prod_{j=1}^r p_j^{s_j}-1\big)(n-1)}.$$ Now, multiply numerator and denominator by $\prod_{j\in J_-} p_j^{-s_j}$, yielding $$\gamma
=\beta\;
\frac{\left(\prod_{j\in J_-} p_j^{-s_j}\right)(n^c-1)}
{\left(\prod_{j\in J_+} p_j^{s_j}
-\prod_{j\in J_-} p_j^{-s_j}\right)(n-1)}.$$ Now, for $(\gamma,c)$ to lie in $\mathrm{Fix}\,\varphi$,  $\gamma$ must belong to $\mathbb{Z}\!\left[\tfrac{1}{n}\right]$. In other words, we must find a value of $c$ that ensures $\gamma\in \mathbb{Z}\!\left[\tfrac{1}{n}\right]$. Let $Q=\bigg(\prod_{j\in J_+ } p_j^{\,s_j}
\;-\;
\prod_{j\in J_-} p_j^{-s_j}\bigg)$. Note that $\gcd(Q,p_j)=1$ for all primes $p_j$ appearing in the factorization of $\alpha$, that is for all $p_j$ where $j\in J_+ \bigcup J_-$. Hence $Q=\bigg(\prod_{j\in J_0} p_j^{\,s_j}\bigg) Q'$,
where the product is taken over all primes $p_j$ that appear in the factorization of $n$ but not in the factorization of $\alpha$, with $s_j\ge 0$. Consequently, $\gcd(Q',n)=1$. Thus, for some $\beta = \frac{l}{n^{q}}$ for some $l,q\in \mathbb{Z}$ with $q\ge 0$, 
\[
\gamma
= \frac{l\,(n^{c}-1)}{n^{q}\;\big(\prod_{j\in J_0} p_j^{\,s_j}\big)\,Q'\,(n-1)}.
\] Note that $n^{q}\;\big(\prod_{j\in J_0} p_j^{\,s_j}\big)$ is invertible in $\mathbb{Z}\!\left[\tfrac{1}{n}\right]$. Also, $\gcd(n-1,n)=1$ and $\gcd(Q',n)=1$ and consequently $(n-1)Q'$ is not invertible. Since
$(n-1)Q'$ is co-prime to $n$, the congruence equation $n^c\equiv 1 \,(\text{mod}\, (n-1)Q')$ has a non-zero solution for $c$ (Euler's totient function, for instance), say $c_0$. For all elements in $\Fix\varphi$, $c$ uniquely determines $\gamma$ (by equation ~\ref{eq:2}). So, we can check for all divisors of $c_0$ whether $\gamma=-\frac{\beta}{\alpha-1} \frac{n^c - 1}{n - 1}$ belongs to $\mathbb{Z}\!\left[\tfrac{1}{n}\right]$. Let $c=\tilde{c}$ be the smallest divisor of $c_0$ for which $-\frac{\beta}{\alpha-1} \frac{n^c - 1}{n - 1}\in \mathbb{Z}\!\left[\tfrac{1}{n}\right]$. Then $\Fix \varphi=\langle(\tilde{\gamma},\tilde{c})\rangle$, where $\tilde{\gamma}=-\frac{\beta}{\alpha-1} \frac{n^{\tilde{c}} - 1}{n - 1}$.

\end{proof}

Now we will talk about stabilizers of $\BS(1,n)$ under the action of its automorphisms. 

\begin{dfn}
    Let $H$ be a subgroup of a group $G$. Then stabilizer of $H$ denoted by $\stab(H)$ is defined by the set $\{\varphi \in \Aut(G) \mid \varphi(h)=h$ for all $h \in H\}.$
\end{dfn}
Clearly, $\stab(H)$ is a subgroup of $\Aut(G).$
Note that if $H=\langle g \rangle$, where $g \in G$ then, $\stab(H)=$$\{\varphi \in \Aut(G) \mid \varphi(g)=g\}$. In this case, we denote $\stab(H)$ as $\stab(g).$ We also observe that $\stab(g)=\stab(g^{-1}).$
\begin{thm} \label{stab}
    Let $(\gamma,c)$ be a non-trivial element of $\BS(1,n)$. Then  $\stab(\gamma,0) \cong \mathbb{Z}\left[\frac{1}{n}\right]$ and if $c \neq 0$, then $\stab(\gamma,c)$ is a finitely generated abelian group with rank equals to the number of distinct prime present in $n$ and a finite set of generators is computable.
\end{thm}
\begin{proof}

  Let $\varphi=\begin{pmatrix} \alpha & \beta \\ 0 & 1 \end{pmatrix} \in \stab(\gamma,c)$ for some $\alpha\in \ZZ\left[\tfrac{1}{n}\right]^*$ and $\beta \in \ZZ\left[\tfrac{1}{n}\right]$ be a non-identity automorphism of $\BS(1,n).$  Then $\varphi(\gamma,c)=(\gamma,c)$ implies $(\alpha-1) \gamma + \beta \frac{n^c - 1}{n - 1}=0,$ which is same as ~\ref{eq:2}. Since $\stab(g)=\stab(g^{-1})$, we can assume $c \geq 0$.\\

  \textit{Case (i) : $c=0$.} \newline
  Putting $c=0$ in ~\ref{eq:2}, we have $(\alpha-1) \gamma=0.$ Since $(\gamma,c)$ is a non-trivial element of $\BS(1,n)$, we have $\alpha=1$. So $\stab(\gamma,0)= \left\{\begin{pmatrix} 1 & \beta \\ 0 & 1 \end{pmatrix} \middle | \beta \in \mathbb{Z}\left[\frac{1}{n}\right]\right\} \cong \mathbb{Z}\left[\frac{1}{n}\right]$, which is an infinitely generated abelian group. \\

  \textit{Case (ii) : $c > 0$.} \newline
  As observed earlier, in order for $\varphi=\begin{pmatrix} \alpha & \beta \\ 0 & 1 \end{pmatrix}$ to be in $\stab(\gamma,c)$, $\alpha, \beta$ must satisfy equation ~\ref{eq:2}. Thus since $\alpha$ uniquely determines the value of $\beta$, we have an injective homomorphism $\stab(\gamma,c) \to \ZZ\left[\tfrac{1}{n}\right]^*$, sending $\varphi=\begin{pmatrix} \alpha & \beta \\ 0 & 1 \end{pmatrix} \mapsto \alpha$. Since $\ZZ\left[\tfrac{1}{n}\right]^* \cong \mathbb{Z}/2\mathbb{Z} \times \mathbb{Z}^r$, where $r$ is the number of distinct primes present in $n$, we conclude that $\stab(\gamma,c)$ is a finitely generated abelian group with rank $\leq r.$ Now, we give an algorithm to compute a finite set of generators for $\stab(\gamma,c)$. In equation ~\ref{eq:2}, let $\frac{n^c-1}{n-1}=\mu$. So, we have $(\alpha-1)\gamma+\mu\beta=0$ where, $\alpha,\beta$ are the unknowns. Since, $\ZZ\left[\tfrac{1}{n}\right]$ is a Euclidean Domain, there exists \(d\in\mathbb{Z}\!\left[\tfrac{1}{n}\right]\) with \(d\) a greatest common divisor of \(\gamma,\mu\) (unique up to a unit), and after dividing by \(d\) and clearing denominators one obtains an equation \begin{equation}\label{eq:3}
     (\alpha-1)\gamma'+\mu'\beta=0  
  \end{equation} with \(\mu'\in\mathbb{Z}\) and \(\gcd(\mu',n)=1\) (as $\gcd(\mu,n)=1$) . Clearly, $(\alpha-1)=D\mu'$ and $\beta=-D\gamma'$ where $D\in \ZZ\left[\tfrac{1}{n}\right]$ are all the solutions of equation ~\ref{eq:3}. Consider the part $(\alpha-1)=D\mu'$. This implies $\alpha=1+D\mu'=1+\frac{x}{n^y}\mu'=\frac{n^y+x\mu'}{n^y}$ where $x,y\in \ZZ$ and $y\ge0$. Since, $\alpha, n^y\in \ZZ\left[\tfrac{1}{n}\right]^*$, $n^y+x\mu'$ should also be invertible. Note that $n^y+x\mu'\in \ZZ$. Therefore we have, $n^y+x\mu'=\prod_{j=1}^r p_j^{z_j}$ where $z_i\ge0$. This implies \begin{equation}\label{eq:4}
    n^y\equiv\prod_{j=1}^r p_j^{z_j}(\text{mod}\, \mu')  
  \end{equation}
  Note that $\gcd(\mu',p_i)=1$ for $1\le i\le r$ which implies $\exists o(p_i)\in \N$ such that $p_i^{o(p_i)}\equiv1 (\text{mod}\, \mu')$ for $1\le i\le r$. Also, note that $\gcd(\mu',n)=1$, which means $\exists o(n)\in \N$ such that $n^{o(n)}\equiv1 (\text{mod}\, \mu')$. Now, for $1\le j\le o(n)-1$, define $$\mathcal{L}^j=\left\{(s_1,..,s_r)\middle| 0\le s_i\le o(p_i)-1 \: \forall 1\le i\le r \,\text{such that}\, \prod_{i=1}^r p_i^{s_i}\equiv n^j(\text{mod}\, \mu')\right\}.$$ We claim that the set $\left\{\begin{pmatrix}
      \alpha &-\tfrac{\gamma}{\mu}(\alpha-1)\\0& 1
  \end{pmatrix}\in \Aut(\BS(1,n))\middle| \alpha\in \mathcal{S}\right\}$ generates $\stab(\gamma,c)$, where $$\mathcal{S}=\left\{p_i^{o(p_i)}\middle| 1\le i\le r\right\}\bigcup \left\{n^{o(n)}\right\}\bigcup\bigcup_{j=0}^{o(n)-1}\left\{\frac{\prod_{i=1}^{r}p_i^{s_i}}{n^j}\middle |(s_1,...,s_r)\in \mathcal{L}^j\right\}.$$ Let $\varphi\in \stab (\gamma,c)\leq \Aut (\BS(1,n))$ where $\varphi=\begin{pmatrix} \alpha & \beta \\ 0 & 1 \end{pmatrix}$ where $\alpha\in \ZZ\left[\tfrac{1}{n}\right]^*$. We know from previous discussion, $\alpha=\frac{n^y+x\mu'}{n^y}$ and from ~\ref{eq:4}, $n^y\equiv\prod_{i=1}^{r}p_i^{z_i} (\text{mod}\, \mu')\implies n^{j+\lambda o(n)}\equiv\prod_{i=1}^{r}p_i^{s_i+\lambda_i o(p_i)} (\text{mod}\, \mu') $ for some $\lambda, \lambda_1,...,\lambda_r\in \ZZ \implies n^j\equiv\prod_{i=1}^{r}p_i^{s_i} (\text{mod}\, \mu')\implies (s_1,...,s_r)\in \mathcal{L}^j$ and $$\alpha=\frac{\prod_{i=1}^{r}p_i^{z_i}}{n^y}=\frac{\prod_{i=1}^{r}p_i^{s_i+\lambda_i o(p_i)}}{n^{j+\lambda o(n)}}=\frac{\prod_{i=1}^{r}p_i^{s_i}}{n^j}\prod_{i=1}^{r}\left (p_i^{o(p_i)}\right)^{\lambda_i}\left(\frac{1}{n^{o(n)}}\right)^\lambda.$$ We also need to verify if the generators stabilize $(\gamma,c)$. We know that any $\varphi=\begin{pmatrix}
      \alpha& \beta\\0 & 1
  \end{pmatrix}\in \stab (\gamma,c)$ if and only if $\alpha,\beta$ satisfies ~\ref{eq:2}. So, to prove our claim, we just need to check, if for $\alpha\in \mathcal{S}$, $\beta=- \frac{\gamma}{\mu}(\alpha-1)\in \ZZ\left[\tfrac{1}{n}\right]$. When $\alpha=\frac{\prod_{i=1}^{r}p_i^{s_i}}{n^j}$, $\beta=- \frac{\gamma}{\mu}\left(\frac{\prod_{i=1}^{r}p_i^{s_i}}{n^j}-1\right)= - \frac{\gamma'}{\mu'}\left(\frac{\prod_{i=1}^{r}p_i^{s_i}-n^j}{n^j}\right) $. Since, $(s_1,..,s_r)\in \mathcal{L}^j$, $\exists x\in \ZZ$ such that $\prod_{i=1}^{r}p_i^{s_i}-n^j=x\mu'\implies x=\frac{\prod_{i=1}^{r}p_i^{s_i}-n^j}{\mu'}\in \ZZ$. Therefore, $\beta=-\frac{\gamma'}{n^j}x\in \ZZ\left[\tfrac{1}{n}\right]$. If $\alpha= p_i^{o(p_i)}$, $\beta=- \frac{\gamma}{\mu}(p_i^{o(p_i)}-1)$. Like in the last case, we have an integer $t=\frac{p_i^{o(p_i)}-1}{\mu'}$ and hence, $\beta=- \gamma't\in \ZZ\left[\tfrac{1}{n}\right] .$ Similar computation shows $\beta\in \ZZ\left[\tfrac{1}{n}\right] $ for $\alpha=n^{o(n)}$. So, the set $\left\{\begin{pmatrix}
      \alpha & - \frac{\gamma}{\mu}(\alpha-1)\\0& 1
  \end{pmatrix}\in \Aut(\BS(1,n))\middle| \alpha\in \mathcal{S}\right\}$ generates $\stab(\gamma,c)$ and is clearly a finite set. Now, we know that the rank of $ \stab (\gamma,c))$ is less than or equal to $ r$. Notice that the subset $\left\{\begin{pmatrix}
   p_i^{o(p_i)}    &- \frac{\gamma}{\mu}(p_i^{o(p_i)}-1)\\0& 1
  \end{pmatrix}\middle|1\le i\le r\right\}$ of the generating set is linearly independent and has rank $r$. Hence, when $c\neq 0$, $\stab(\gamma,c)$ is an abelian group with rank $r$.
\end{proof}

\begin{cor}
    Let $H$ be a finitely generated subgroup of $\BS(1,n)$. Then $\stab(H)$ is either trivial or $\stab(H) \cong \mathbb{Z}\left[\frac{1}{n}\right]$ or it is a finitely generated abelian group with rank equals to the number of distinct primes present in $n$ and in this case, a finite set of generators is computable.
\end{cor}

\begin{proof}
    Let $H$ be a finitely generated subgroup of $\BS(1,n)$. Then by Lemma ~\ref{subgroups}, there are two cases to consider. \\

    \textit{Case (i) : $H$ is finite index in $\BS(1,n).$} \newline
    Let $\varphi \in \stab(H),$ then $H \leq \Fix(\varphi) \leq \BS(1,n).$ Since $H$ is finite index in $\BS(1,n)$, $\Fix(\varphi)$ is also finite index in $\BS(1,n).$ But by Theorem ~\ref{fixed}, we have that either $\Fix\varphi\cong\ZZ$ or $\Fix\varphi\cong \ZZ \left[\tfrac{1}{n}\right]$, neither of which are finite index in $\BS(1,n).$ So $\Fix(\varphi)=\BS(1,n)$, which implies that $\varphi$ is the identity automorphism. Hence, in this case, $\stab(H)$ is trivial. \\

    \textit{Case (ii) : $H=\langle g\rangle$ is infinite cyclic, where $g \in \BS(1,n).$} \newline
    In this case, $\stab(H)=\stab(g)$ and by Theorem ~\ref{stab},  it is $\cong \mathbb{Z}\left[\frac{1}{n}\right]$ or it is a finitely generated abelian group with rank equals to the number of distinct primes present in $n$ and in this case, a finite set of generators is computable.
    
\end{proof}

Following the terminology from \cite{RV}, the auto-fixed closure of $H$ in $G$, 
denoted $\mathrm{Cl}(H)$, is the subgroup of $G$ defined by
\[
\mathrm{Cl}(H)
= \Fix(\stab(H))
= \bigcap_{\substack{\varphi \in \Aut(G) \\ H \le \Fix \varphi}}
\Fix \varphi .
\]

\begin{thm}\label{afix}
Let $H$ be a finitely generated subgroup of $\BS(1,n)$. \\
(i) If $\stab(H)$ is trivial, then $\mathrm{Cl}(H)=\BS(1,n)$ \newline
(ii) If $\stab(H) \cong \mathbb{Z}\left[\frac{1}{n}\right]$, then $\mathrm{Cl}(H)=\left\{(\gamma,0) \mid \gamma \in \mathbb{Z}\left[\frac{1}{n}\right]\right \}$ \newline
(iii) If $\stab(H)$ is a finitely generated abelian group, then $\mathrm{Cl}(H) \cong \mathbb{Z}$, and a generator is computable.
\end{thm}

\begin{proof}
    We will consider a few cases. \\

    \textit{Case (i) : $\stab(H)$ is trivial.} \newline
    In this case, \[
\mathrm{Cl}(H)
= \Fix(\stab(H))
=\Fix(\text{Id})
=\BS(1,n).
\] \\

\textit{Case (ii) : $\stab(H) \cong \mathbb{Z}\left[\frac{1}{n}\right]$.} \newline
From the proof of ~\ref{stab}, we observe that $H \leq H',$ where $H'=\left\{(\gamma,0) \middle| \gamma \in \mathbb{Z}\left[\frac{1}{n}\right]\right\}$ and $\stab(\gamma,0)= \left\{\begin{pmatrix} 1 & \beta \\ 0 & 1 \end{pmatrix} \middle | \beta \in \mathbb{Z}\left[\frac{1}{n}\right]\right\} \cong \mathbb{Z}\left[\frac{1}{n}\right]$, for any $\gamma \in \mathbb{Z}\left[\frac{1}{n}\right]$. So \[
\mathrm{Cl}(H)
= \Fix(\stab(H))
= \Fix\left(\left\{\begin{pmatrix} 1 & \beta \\ 0 & 1 \end{pmatrix} \middle | \beta \in \mathbb{Z}\left[\frac{1}{n}\right]\right\}\right)=H'.
\] \\

\textit{Case (iii) : $\stab(H)$ is a finitely generated abelian group with rank equals to the number of distinct primes present in $n$.} \newline
In this case $\stab(H)=\left\langle \begin{pmatrix} \alpha_1 & \beta_1 \\ 0 & 1 \end{pmatrix},\ldots,\begin{pmatrix} \alpha_r & \beta_r \\ 0 & 1 \end{pmatrix}\right \rangle$, where $\alpha_i \neq 1 \forall 1 \leq i \leq r.$
So \[
\mathrm{Cl}(H)
= \Fix(\stab(H))
= \bigcap_{\substack{1 \leq i \leq r}}\Fix\left(\begin{pmatrix} \alpha_i & \beta_i \\ 0 & 1 \end{pmatrix}\right)
\]
Now, by Theorem ~\ref{fixed}, we have that $\Fix\left(\begin{pmatrix} \alpha_i & \beta_i \\ 0 & 1 \end{pmatrix}\right) \cong \mathbb{Z}$ and a generator $(\gamma_i,c_i)$ is computable. Then  \[
\mathrm{Cl}(H)
= \Fix(\stab(H))
= \bigcap_{\substack{1 \leq i \leq r}}\Fix\left(\begin{pmatrix} \alpha_i & \beta_i \\ 0 & 1 \end{pmatrix}\right) = \langle (\gamma_k, c_k) \rangle \cong \mathbb{Z},  
\] where $c_k=\lcm(c_1,...,c_r)$ and $\gamma_k$ is uniquely determined by $c_k$. 
\end{proof}

We have a similar result as Theorem ~\ref{per} for stabilizers.

\begin{thm}\label{stab}
    For $(\gamma,c) \in \BS(1,n)$, $\stab(\gamma,c)=\stab(\gamma,c)^k,$ for all $k \geq 1.$
\end{thm}

\begin{proof}
    We will consider two cases. \\

    \textit{Case (i) : $c=0.$} \newline
    Note that by Lemma \ref{basic}, $(\gamma,0)^k=(k\gamma,0)$. By the proof of ~\ref{stab}, we saw that $\stab(\gamma,0)= \left\{\begin{pmatrix} 1 & \beta \\ 0 & 1 \end{pmatrix} \middle | \beta \in \mathbb{Z}\left[\frac{1}{n}\right]\right\}$, for all $\gamma \in \mathbb{Z}[\frac{1}{n}].$ Hence, $\stab(\gamma,0)=\stab(\gamma,0)^k,$ for all $k \geq 1.$ \\

    \textit{Case (ii) : $c>0.$} \newline
    Let $\varphi \in \stab(\gamma,c) \implies \varphi(\gamma,c)=(\gamma,c) \implies \varphi^k(\gamma,c)=(\gamma,c)$, for all $k \geq 1 \implies (\gamma,c) \in \stab(\gamma,c)^k,$ for all $k \geq 1.$ Hence, $\stab(\gamma,c) \subseteq \stab(\gamma,c)^k,$ for all $k \geq 1.$ \newline
    Now let $\varphi=\begin{pmatrix} \alpha & \beta \\ 0 & 1 \end{pmatrix} \in \stab(\gamma,c)^k$. Again by Lemma \ref{basic}, $(\gamma,c)^k=\left(\gamma\frac{n^{ck}-1}{n^c-1}, kc\right).$ Now, $\begin{pmatrix} \alpha & \beta \\ 0 & 1 \end{pmatrix}\left(\gamma\frac{n^{ck}-1}{n^c-1}, kc\right)=\left(\gamma\frac{n^{ck}-1}{n^c-1}, kc\right) \implies \alpha \gamma\frac{n^{ck}-1}{n^c-1}+ \beta \frac{n^{ck}-1}{n-1}=\gamma\frac{n^{ck}-1}{n^c-1}$. Since $ck \neq 0$, we have, $\alpha \gamma\frac{1}{n^c-1}+ \beta \frac{1}{n-1}=\gamma\frac{1}{n^c-1}$. Multiplying both sides by $n^c-1,$ we have $\alpha \gamma+ \beta \frac{n^c-1}{n-1}=\gamma \implies \begin{pmatrix} \alpha & \beta \\ 0 & 1 \end{pmatrix} (\gamma,c)=(\gamma,c) \implies \varphi \in \stab(\gamma,c) \implies \stab(\gamma,c)^k \subseteq \stab(\gamma,c).$ Hence, $\stab(\gamma,c)=\stab(\gamma,c)^k,$ for all $k \geq 1.$
    
    \end{proof}

The following corollary is an interesting consequence of the above theorem. Although, the result is previously known (see \text{Lemma 2.5}, \cite{Bo}), we give a different proof below.
\begin{cor}
    $\BS(1,n)$ has the unique root property, that is, if for any $g_1,g_2\in \BS(1,n)$ and any $k\in \N$, $g_1^k=g_2^k\implies g_1=g_2$. 
\end{cor}
\begin{proof}
 Suppose $\exists g_1,g_2\in \BS(1,n)$ such that $g_1^k=g_2^k=g$ for some $k\in \N$. By theorem ~\ref{stab}, we have $\stab (g_1)=\stab(g_1^k)$ and $\stab (g_2)=\stab(g_2^k)$ and therefore, $\stab (g_1)=\stab (g_2)$. Let $i_{g_1}$ denote the inner automorphism $i_{g_1}(x)=g_1xg_1^{-1}$ and note that $i_{g_1}\in \stab(g_1)$. This implies $i_{g_1}\in \stab(g_2)\implies i_{g_1}(g_2)=g_2\implies g_1g_2g_1^{-1}=g_2\implies g_1g_2=g_2g_1 $. Now, consider $(g_1g_2^{-1})^k=(g_1)^k(g_2)^{-k}=gg^{-1}=1$. Since, $\BS (1,n)$ is torsion-free, we have $g_1g_2^{-1}=1\implies g_1=g_2$.  
\end{proof}

\section{Fixed subgroups of endomorphisms and endo-stabilizers}

\begin{dfn}
Let $G$ be a group and $\End(G)$ be the monoid of its endomorphisms. Let $\varphi\in \End(G)$, then the fixed subgroup of $\varphi$ denoted by $\Fix\varphi$ is given by $\Fix\varphi=\left\{g\in G\middle| \varphi(g)=g\right\}$  
\end{dfn}

Clearly, $\Fix\varphi$ forms a group. The following theorem is an analog of Theorem \ref{fixed} in the case of endomorphisms. 

\begin{thm}\label{e-fix}
    Let $\varphi \in \End (\BS(1,n))$ be a non-identity endomorphism. Then $\Fix\varphi\cong\ZZ$ if $\varphi\in \End^{\mathrm{I}}(\BS(1,n))$ and $\alpha\neq 1$ or if $\varphi\in \End^{\mathrm{II}}(\BS(1,n))$ and $\varphi=\varphi_{(\gamma,1)}$ where $\gamma\in \ZZ\left[\frac{1}{n}\right] $, in which case $\Fix \varphi= \langle(\gamma,1)\rangle$. When $\varphi\in \End^{\mathrm{I}}(\BS(1,n))$ and $\alpha=1$, $\Fix\varphi\cong\ZZ\left[\frac{1}{n}\right]$ and for any $\varphi\in \End^{\mathrm{II}}(\BS(1,n)) $ not of the type $\varphi_{(\gamma,1)}$ , $\Fix\varphi=\{(0,0)\}$.
\end{thm}
 \begin{proof}
 \textit{Case i}: When $\varphi\in \End^{\mathrm{I}}(\BS(1,n))$ and $\alpha= 1$, the proof is the same as in the case for $\varphi\in \Aut(\BS(1,n))$. Therefore, we have $\Fix\varphi\cong\ZZ\left[\frac{1}{n}\right]$. \\
 
 \textit{Case~ii.} Suppose $\varphi \in \End^{\mathrm{I}}(\BS(1,n))$ and $\alpha \neq 1$.
The proof is similar to the case $\varphi \in \Aut(\BS(1,n))$; the only difference is that
here $\alpha \in \mathbb{Z}\!\left[\tfrac{1}{n}\right]$ rather than $\mathbb{Z}\!\left[\tfrac{1}{n}\right]^{\!*}$. By equation~\eqref{eq:2}, the map $\BS(1,n) \longrightarrow \mathbb{Z}, \qquad (\gamma,c) \longmapsto c$, when restricted to $\Fix \varphi$, is injective, and hence $\Fix \varphi \leq \mathbb{Z}$. For an element $(\gamma,c)$ to lie in $\Fix \varphi$, equation~\eqref{eq:2} gives
$\gamma = -\frac{\beta}{\alpha-1}\,\frac{n^{c}-1}{n-1}$ and this must belong to $\mathbb{Z}\!\left[\tfrac{1}{n}\right]$. Thus we must determine values of $c$ for which $\gamma \in \mathbb{Z}\!\left[\tfrac{1}{n}\right]$. Since $\alpha \in \mathbb{Z}\!\left[\tfrac{1}{n}\right]$, we may write
$\alpha = \tfrac{l}{n^{p}}$ for some $l,p \in \mathbb{Z}$ with $p>0$, and similarly
$\beta = \tfrac{m}{n^{q}}$ for some $m,q \in \mathbb{Z}$ with $q>0$. Substituting these expressions for $\alpha$ and $\beta$, we obtain $\gamma
= -\frac{(n^{c}-1) m n^{p}}{(n-1)(n^{p}-l)n^{q}}
= -\frac{(n^{c}-1) m n^{p-q}}{\big(\prod_{j} p_j^{\,s_j}\big)(n-1)A},$ where we write
$n^{p}-l = \prod_{j} p_j^{\,s_j} A$, with $p_j$ ranging over the prime divisors of $n$, $s_j \ge 0$, and $\gcd(n,A)=1$. Using the same argument as in the automorphism case, we conclude that
$\Fix \varphi \cong \mathbb{Z}$.\\

\textit{Case iii}: When $\varphi \in \End^{\mathrm{II}}(\BS(1,n))$, then $\varphi=\varphi_{(\gamma,c)}$ for some $(\gamma,c)\in \BS(1,n)$. If $(\gamma',c')\in \Fix\varphi$, then $\varphi_{(\gamma,c)}(\gamma',c')=(\gamma',c')\implies\left(\tfrac{n^{cc'}-1}{n^c-1}\gamma,cc'\right)=(\gamma',c')\implies c=1$ and $\tfrac{n^{cc'}-1}{n^c-1}\gamma=\gamma'$. Firstly, this means $\Fix\varphi\neq \emptyset\implies c=1$, that is, if $\varphi=\varphi_{(\gamma,1)}$ for some $\gamma\in \mathbb{Z}\!\left[\tfrac{1}{n}\right] $. This also means $\left(\tfrac{n^{c'}-1}{n-1}\gamma,c'\right)\in \Fix\varphi$ for all $c'\in\ZZ$, in particular, $(\gamma,1)\in \Fix\varphi$. Therefore, $\Fix\varphi=\langle(\gamma,1)\rangle$ as $c'$ uniquely determines $\gamma'$ by \ref{eq:2}. So, we conclude $\Fix\varphi\neq \emptyset$ if and only if $\varphi=\varphi_{(\gamma,1)}$ for some $\gamma\in \mathbb{Z}\!\left[\tfrac{1}{n}\right] $, in which case, $\Fix\varphi=\langle(\gamma,1)\rangle\cong \ZZ$.
     
 \end{proof}

\begin{dfn}
   Let $G$ be a group and $\End(G)$ be the monoid of endomorphisms of $G$. Let $\varphi\in \End(G)$, we define its periodic subgroup as \[
\operatorname{Per}(\varphi)=\bigcup_{k=1}^\infty \operatorname{Fix}(\varphi^k).
\qquad
\]
\end{dfn}

Note that $\operatorname{Per}(\varphi)$ is indeed a subgroup $\text{since } x\in \Fix(\varphi^{k}) \text{ and } y\in \Fix(\varphi^{k'}) 
\text{ imply } xy\in \Fix(\varphi^{kk'})$.

\begin{thm} \label{per}
  Let $\varphi\in \End(\BS(1,n))$ be an endomorphism, then $\operatorname{Per}(\varphi)=\operatorname{Fix}(\varphi)$. In fact, $\operatorname{Fix}(\varphi)=\operatorname{Fix}(\varphi^k), \forall k \geq 1.$
\end{thm}

\begin{proof}
Clearly, the above statement is true when $\phi$ is the identity map.
Now we will consider separate cases for \textit{Type (i)} and \textit{Type (ii)} endomorphisms. \newline
\textit{Case (i) : $\varphi=\begin{pmatrix} \alpha & \beta \\ 0 & 1 \end{pmatrix} \in \End^{\mathrm{I}}(\BS(1,n))$ for some $\alpha\in \ZZ\left[\tfrac{1}{n}\right] \setminus{\{0\}},\, \beta \in \ZZ\left[\tfrac{1}{n}\right]$ is a non-identity endomorphism.}
    
    We will consider two subcases. \newline

    \textit{Subcase (i) : $\alpha =1.$} \newline
    Since $\varphi$ is a non-identity endomorphism, $\beta \neq 0.$ An easy calculation shows that if $\varphi=\begin{pmatrix} 1 & \beta \\ 0 & 1 \end{pmatrix}$, then $\varphi^k=\begin{pmatrix} 1 & k\beta \\ 0 & 1 \end{pmatrix} $. Let $(\gamma,c) \in \operatorname{Fix}(\varphi^k)$. Then $\begin{pmatrix} 1 & k\beta \\ 0 & 1 \end{pmatrix}(\gamma,c)=(\gamma,c)$, which implies $ k\beta \frac{n^c - 1}{n - 1}=0$. Since $k, \beta \neq 0,$ we have $c=0.$ So $\operatorname{Fix}(\varphi^k)=\left\{(\gamma,0)\middle|\gamma\in \ZZ\left[\tfrac{1}{n}\right]\right\}$, for all $k \in \mathbb{N}$. Thus $\operatorname{Per}(\varphi)=\left\{(\gamma,0)\middle|\gamma\in \ZZ\left[\tfrac{1}{n}\right]\right\}=\operatorname{Fix}(\varphi).$ \newline
   \textit{Subcase (ii) : $\alpha \neq 1.$} \newline
   In this case, using a simple induction argument, one can show that $\varphi^k=\begin{pmatrix} \alpha^k & \beta\frac{\alpha^k - 1}{\alpha - 1} \\ 0 & 1 \end{pmatrix}$. Now if $(\gamma,c) \in \Fix(\varphi^k)$, then $\begin{pmatrix} \alpha^k & \beta\frac{\alpha^k - 1}{\alpha - 1} \\ 0 & 1 \end{pmatrix}(\gamma,c)=(\gamma,c),$ which implies \newline
   $(\alpha^k-1)\gamma=- \frac{n^c - 1}{n - 1}\beta\frac{\alpha^k - 1}{\alpha - 1}$. Since $k \geq 1,$ we have $\gamma=- \frac{\beta}{\alpha - 1}\frac{n^c - 1}{n - 1}$. 

   Note that in the proof of last theorem (case (ii)), we saw that $\left(- \frac{\beta}{\alpha - 1}\frac{n^c - 1}{n - 1},c\right) \in \Fix(\varphi)$. So we have, $\Fix(\varphi^k) \subset \Fix(\varphi)$ for all $k \in \mathbb{N}$ and $\Fix(\varphi) \subset \Fix(\varphi^k)$ always holds for all $k \in \mathbb{Z}.$ Hence, $\Fix(\varphi^k)=\Fix(\varphi)$ for all $k \in \mathbb{N}$, that is $\operatorname{Per}(\varphi)=\operatorname{Fix}(\varphi)$. \newline

   \textit{Case (ii) : $\varphi=\varphi_{(\gamma,c)}$$ \in \End^{\mathrm{II}}(\BS(1,n))$, for some $(\gamma,c)\in \BS(1,n)$.} \newline
   From the composition rule of such endomorphisms described in section \S 2, we note that $\varphi^k=\varphi_{(\frac{n^{c^k}-1}{n^c-1}\gamma,c^k)}$
   Again we will consider two subcases.

   Subcase (i) : c=1 \newline
   Putting $c=1,$ we get $\varphi^k=\varphi_{(\gamma,c)}=\varphi$. Then obviously, $\Fix(\varphi^k)=\Fix(\varphi)$ for all $k \in \mathbb{N}$, that is $\operatorname{Per}(\varphi)=\operatorname{Fix}(\varphi).$

   Subcase (ii) : $c \neq 1$ \newline
   In this case, $c^k \neq 1$. So $\varphi^k = \varphi_{(\frac{n^{c^k}-1}{n^c-1}\gamma,c^k)} \neq \varphi_{(\gamma,1)}.$ Thus by the last theorem, $\Fix\varphi^k=\{(0,0)\}=\Fix\varphi$ for all $k \in \mathbb{N}$. So $\operatorname{Per}(\varphi)=\operatorname{Fix}(\varphi)$.

\end{proof}

\begin{rem}
    In particular, if $\varphi$ is an automorphism of $\BS(1,n),$ then $\operatorname{Fix}(\varphi)=\operatorname{Fix}(\varphi^k), \forall k \geq 1$ and $\operatorname{Per}(\varphi)=\operatorname{Fix}(\varphi)$.
\end{rem}
\noindent Now we talk about stabilizers of elements in $\BS(1,n)$ as a subset of $\End(G)$ as opposed to $\Aut(G)$ in the previous section.
\begin{dfn}

 Let $G$ be a group. For any element $g\in G$, the endo-stabilizer of $g$ denoted by $e\text{-}\stab (g)$ is defined as $e\text{-}\stab (g)=\left\{\varphi\in \End (G)\middle | \varphi(g)=g\right\}.$
 \end{dfn}
 Note that $e\text{-}\stab (g)$ is a submonoid of $\End (G)$. Since  $\End^{\mathrm{I}}(\BS(1,n))$ and $\End^{\mathrm{II}}(\BS(1,n))$ are two disjoint subsets of the monoid $\End(\BS(1,n))$, we have $$e\text{-}\stab(\gamma, c)=\stab^\mathrm{I}(\gamma,c)\bigcup \stab^\mathrm{II}(\gamma,c)$$ where $$\stab^\mathrm{I}(\gamma,c)=\left\{\varphi\in \End^{\mathrm{I}}(\BS(1,n))\middle | \varphi(\gamma,c)=(\gamma,c)\right\}$$ and $$\stab^\mathrm{II}(\gamma,c)=\left\{\varphi\in \End^{\mathrm{II}}(\BS(1,n))\middle | \varphi(\gamma,c)=(\gamma,c)\right\}$$ are disjoint subsets of $\End^{\mathrm{I}}(\BS(1,n))$ and $\End^{\mathrm{II}}(\BS(1,n))$ respectively.

 \begin{dfn}
    Let $H$ be a subgroup of a group $G$. Then stabilizer of $H$ denoted by $e\text{-}\stab(H)$ is defined by the set $\{\varphi \in \End(G) \mid \varphi(h)=h$ for all $h \in H\}.$
\end{dfn}
Note that $e\text{-}\stab(H)$ is a submonoid of $\End(G).$
Note that if $H=\langle g \rangle$, where $g \in G$ then, $e\text{-}\stab(H)=$$\{\varphi \in \End(G) \mid \varphi(g)=g\}$. In this case, we denote $e\text{-}\stab(H)$ as $e\text{-}\stab(g).$ We also observe that $e\text{-}\stab(g)=e\text{-}\stab(g^{-1}).$

 \begin{thm}\label{e-stab}
Let $(\gamma,c)\in \BS(1,n)$ where $\gamma=\frac{l}{n^p}$, $l,p\in \ZZ$ with $p\ge0,l\neq0$ and $\mu=\frac{n^c-1}{n-1}$. Then, $e\text{-}\stab(\gamma,c)$ is an infinitely generated submonoid of $\End(BS(1,n))$. In particular for $c>0$, $\stab^{\mathrm{I}}(\gamma,c)\cong \ZZ\left[\tfrac{1}{n}\right]$ and $\stab^\mathrm{II}(\gamma,c)=\left\{\varphi_{\left({\frac{\gamma}{\mu}},1\right)}\right\}$ if $\mu \mid l$ and $\stab^\mathrm{II}(\gamma,c)= \emptyset$ if $\mu \nmid l$
 \end{thm}

 \begin{proof} We will consider two cases.
 
 \textit{Case (i):} When $c=0$, equation \ref{eq:2} becomes $(\alpha-1)\gamma=0$ which implies $\alpha=1$ as $\gamma\neq 0$. Then $\stab^{\mathrm{I}}(\gamma,0)=\left\{\begin{pmatrix}
     1 & \beta\\0 &1
 \end{pmatrix}\middle|\beta\in \ZZ\left[\tfrac{1}{n}\right]\right\}\cong \ZZ\left[\tfrac{1}{n}\right]$. If $\varphi\in \stab^{\mathrm{II}}(\gamma,0)$, then $\varphi=\varphi_{(\gamma',c')}$ for some $(\gamma',c')\in \BS(1,n)$ and $\varphi_{(\gamma',c')}(\gamma,0)=(\gamma,0)\implies \left(\frac{n^{0.c'}-1}{n^{c'}-1}\gamma',0.c'\right)=(\gamma,0)$. This implies $\gamma=0$, a contradiction. Therefore, $\stab^{\mathrm{II}}(\gamma,0)=\emptyset$.\\

 \textit{Case (ii):} If $c>0$ and $\varphi\in \stab^\mathrm{I}(\gamma,c) $, then $\varphi\in \End^\mathrm{I}(\BS(1,n))$ and hence $\varphi=\begin{pmatrix}
     \alpha & \beta\\ 0 & 1
 \end{pmatrix}$   for $\alpha,\beta\in \ZZ\left[\tfrac{1}{n}\right]$. Since, $\varphi\in \stab^\mathrm{I}(\gamma,c) $, $\alpha, \beta$ satisfy \ref{eq:2}. As in the case of automorphisms, $\alpha$ uniquely determines $\beta$. Substituting $\mu=\frac{n^c-1}{n-1}$ in \ref{eq:2}, we get, $(\alpha-1)\gamma+\mu\beta=0\implies 1-\alpha=\frac{\mu}{\gamma}\beta$. We need $\alpha\in \ZZ\left[\tfrac{1}{n}\right]$. But $ \alpha\in \ZZ\left[\tfrac{1}{n}\right] \iff 1-\alpha\in  \ZZ\left[\tfrac{1}{n}\right]$, which in turn implies $\frac{\mu}{\gamma}\beta \in \ZZ\left[\tfrac{1}{n}\right]$. Let $\beta=\frac{x}{n^y}$ where $x,y\in \ZZ$ and $y\ge 0$ are unknowns. Substituting all the values in $\frac{\mu}{\gamma}\beta $, we get $\frac{\mu x n^b}{ln^y}$. Since, $n^{b-y}$ is a unit in $\ZZ\left[\tfrac{1}{n}\right]$, $\frac{\mu}{\gamma}\beta \in \ZZ\left[\tfrac{1}{n}\right]\iff \frac{\mu x}{l}\in \ZZ\left[\tfrac{1}{n}\right]$. Let
$l=\prod_{j=1}^{k}p_j^{s_j}A$ where $s_j\ge 0$ and $\gcd(n,A)=1$. Since, $p_j$s are primes appearing in the factorization of $n$, $\prod_{j=1}^{k}p_j^{s_j}$ is an invertible element in $\ZZ\left[\tfrac{1}{n}\right]$. Therefore, $\frac{\mu x}{l}\in \ZZ\left[\tfrac{1}{n}\right]\iff \frac{\mu x}{A}\in \ZZ\left[\tfrac{1}{n}\right]$. Note that $\mu,A\in \ZZ$ such that $\gcd(\mu,n)=1=\gcd(A,n)$ and $x\in \ZZ$ is unknown. Hence, $\frac{\mu x}{A}\in \ZZ\left[\tfrac{1}{n}\right]\iff \frac{\mu x}{A}\in \ZZ$. Then, $x=\frac{kA}{\gcd(\mu,A)}$ for any $k\in \ZZ$ ensures $\alpha\in \ZZ\left[\tfrac{1}{n}\right] $.  Substituting the value of $x$ in $\alpha=1-\frac{\mu}{\gamma}\beta$, we get, $\alpha=1-\frac{\mu}{\gamma}\frac{kA}{\gcd (\mu, A)n^y}$ where $\beta=\frac{kA}{\gcd (\mu, A)n^y}$. It is easy to verify that $\begin{pmatrix}
    1-\frac{\mu}{\gamma}\frac{kA}{\gcd (\mu, A)n^y} & \frac{k A}{\gcd (\mu, A)n^y}\\
    0 &1
\end{pmatrix}$ indeed fixes $(\gamma,c)$ for all values of $k,y\in \ZZ$ with $y\ge 0$. Hence we have $$\stab^\mathrm{I}(\gamma,c)=\left\{\begin{pmatrix}
    1-\frac{\mu}{\gamma}\frac{kA}{\gcd (\mu, A)n^y} & \frac{kA}{\gcd (\mu, A)n^y}\\
    0 &1
\end{pmatrix}\middle |k\in \ZZ, y\ge0\right\}\cong \frac{A}{\gcd(\mu, A)}\ZZ\left[\tfrac{1}{n}\right]\cong \ZZ\left[\tfrac{1}{n}\right].$$ Let us move on to the computation of $\stab^\mathrm{II}(\gamma,c)$. If $\varphi\in \stab^\mathrm{II}(\gamma,c)$, $\varphi\in \End^{\mathrm{I}}(\BS(1,n)) $ and $\varphi(\gamma,c)=(\gamma,c)$. That is, $\varphi=\varphi_{(\gamma',c')}$ for some $(\gamma',c')\in \BS(1,n) $ and $\varphi_{(\gamma',c')}(\gamma,c)=(\gamma,c)\implies \left(\frac{n^{cc'}-1}{n^{c'}-1}\gamma,cc'\right)=(\gamma,c)\implies c'=1$ and $\frac{n^{cc'}-1}{n^{c'}-1}\gamma'=\gamma$. Let $\frac{n^{c}-1}{n-1}=\mu$, then $\gamma'=\frac{\gamma}{\mu}=\frac{l}{n^p\mu}$ must belong to $\ZZ\left[\tfrac{1}{n}\right]$ for $\varphi_{(\gamma',c')}$ to belong to $\stab^\mathrm{II}(\gamma,c)$. As, $n^p$ is invertible in $\ZZ\left[\tfrac{1}{n}\right]$ and $\gcd(\mu,n)=1$, $\frac{l}{n^p\mu}\in \ZZ\left[\tfrac{1}{n}\right] \iff \frac{l}{\mu}\in \ZZ\left[\tfrac{1}{n}\right] $. Hence, $\stab^\mathrm{II}(\gamma,c)\neq \emptyset\iff \mu \mid l$. If $\mu \mid l$, $\stab^\mathrm{II}(\gamma,c)=\left\{\varphi_{\left({\frac{\gamma}{\mu}},1\right)}\right\}$  and $\stab^\mathrm{II}(\gamma,c)= \emptyset$ if $\mu \nmid l$. 
 \end{proof}

\begin{cor}\label{stabH}
    Let $H$ be a finitely generated subgroup of $\BS(1,n)$. Then $e\text{-}\stab(H)$ is either trivial or an infinitely generated submonoid of $\End(\BS(1,n))$
\end{cor}

\begin{proof}
    Let $H$ be a finitely generated subgroup of $\BS(1,n)$. Then by Lemma ~\ref{subgroups}, there are two cases to consider. \\

    \textit{Case (i) : $H$ is finite index in $\BS(1,n).$} \newline 
    The proof is same as that in the case of automorphisms. We write it for completion.
    Let $\varphi \in e\text{-}\stab(H),$ then $H \leq \Fix(\varphi) \leq \BS(1,n).$ Since $H$ is finite index in $\BS(1,n)$, $\Fix(\varphi)$ is also finite index in $\BS(1,n).$ But by Theorem ~\ref{e-fix}, we have that either trivial or $\Fix\varphi\cong\ZZ$ or $\Fix\varphi\cong \ZZ \left[\tfrac{1}{n}\right]$, none of which is finite index in $\BS(1,n).$ So $\Fix(\varphi)=\BS(1,n)$, which implies that $\varphi$ is the identity endomorphism. Hence, in this case, $e\text{-}\stab(H)$ is trivial. \\

    \textit{Case (ii) : $H=\langle g\rangle$ is infinite cyclic, where $g \in \BS(1,n).$} \newline
    In this case, $e\text{-}\stab(H)=e\text{-}\stab(g)$ and by Theorem ~\ref{e-stab},  it is an infinitely generated submonoid of $\End (\BS(1,n)$. 
    
\end{proof}

As in the case of automorphisms, we can define, the endo-fixed closure of $H$ in $G$ \cite{RV}, 
denoted $e\text{-}\mathrm{Cl}(H)$, is the subgroup of $G$ defined by
\[
e\text{-}\mathrm{Cl}(H)
= \Fix(e\text{-}\stab(H))
= \bigcap_{\substack{\varphi \in \End(G) \\ H \le \Fix \varphi}}
\Fix \varphi .
\]
Analogous to theorem \ref{afix}, we have the following theorem for endo-fixed closures.
\begin{thm}
    Let $H$ be a finitely generated subgroup of $\BS(1,n)$. Then $e\text{-}\mathrm{Cl}(H)=\BS(1,n)$ or $e\text{-}\mathrm{Cl}(H)\cong \mathbb{Z}[\frac{1}{n}]$ or $e\text{-}\mathrm{Cl}(H)\cong \mathbb{Z}$.
\end{thm}
\begin{proof}
We will consider two cases.

    \textit{Case (i) : When $H$ is finite index, by corollary \ref{stabH}, $e\text{-}\stab(H)$ is trivial.}
    In this case, \[
e\text{-}\mathrm{Cl}(H)
= \Fix(e\text{-}\stab(H))
=\Fix(\text{Id})
=\BS(1,n).
\] \\
\textit{Case (ii)} : When $H$ is infinite cyclic, by corollary \ref{stabH}, $e\text{-}\stab(H)=e\text{-}\stab(g)$ for some $g\in \BS(1,n)$

Subcase (i) : $g=(\gamma,c),$ where $c=0.$
In this case, \[
e\text{-}\mathrm{Cl}(H)
= \Fix(e\text{-}\stab(\gamma,0))= \Fix\left(\left\{\begin{pmatrix}
     1 & \beta\\0 &1
 \end{pmatrix}\middle|\beta\in  \ZZ\left[\tfrac{1}{n}\right]\right\}\right) \] by proof of case (i) in \ref{e-stab}
\[ =\left\{(\gamma,0)\middle|\gamma\in \ZZ\left[\tfrac{1}{n}\right]\right\}\cong \ZZ\left[\tfrac{1}{n}\right] \] by the proof of case (i) in \ref{fixed}.

Subcase (ii) : $g=(\gamma,c),$ where $c \neq 0.$
 \[
e\text{-}\mathrm{Cl}(H)
= \Fix((e\text{-}\stab(\gamma,c))= \bigcap_{\substack{\varphi \in e\text{-}\stab(\gamma,c)}}
\Fix \varphi .
\]

Note that if $\varphi \in e\text{-}\stab(\gamma,c)$ is non-trivial, then $\Fix(\varphi) \ncong \mathbb{Z}[\frac{1}{n}].$ Because if $\Fix(\varphi) \cong \mathbb{Z}[\frac{1}{n}]$ and $\varphi \in e\text{-}\stab(\gamma,c)$, then by Theorem \ref{e-fix}, $\varphi \in \stab^{\mathrm{I}}(\BS(1,n))=\left\{\begin{pmatrix}
    1-\frac{\mu}{\gamma}\frac{kA}{\gcd (\mu, A)n^y} & \frac{kA}{\gcd (\mu, A)n^y}\\
    0 &1
\end{pmatrix}\middle |k\in \ZZ, y\ge0\right\}$ and $\varphi = \begin{pmatrix}
     1 & \beta\\0 &1
 \end{pmatrix}$. This implies $\frac{\mu}{\gamma}\frac{kA}{\gcd (\mu, A)n^y}=0$. Now $c \neq 0$ implies $\frac{\mu}{\gamma} \neq 0.$ So $\frac{kA}{\gcd (\mu, A)n^y}=0$. So $\varphi$ must be $\begin{pmatrix}
     1 & 0\\0 &1
 \end{pmatrix},$ a contradiction. 

Now, since $ \langle (\gamma,c) \rangle \leq \Fix(\varphi)$ for all $\varphi \in e\text{-}\stab(\gamma,c)$. So $\mathbb{Z} \leq \Fix(\varphi)$ and $\Fix(\varphi) \ncong \mathbb{Z}[\frac{1}{n}]$ for all $\varphi \in e\text{-}\stab(\gamma,c)$. This implies $e\text{-}\mathrm{Cl}(H)=\bigcap_{\substack{\varphi \in e\text{-}\stab(\gamma,c)}}
\Fix \varphi \cong \mathbb{Z}.$

\end{proof}

\begin{thm}
 For any $(\gamma,c)\in \BS(1,n)$, $e\text{-}\stab(\gamma,c)=e\text{-}\stab(\gamma,c)^k$ for all $k\in \ZZ$.   
\end{thm}
\begin{proof}
    Since $e\text{-}\stab(g)=e\text{-}\stab (g^{-1})$, it is enough to prove the theorem for $k>0$. Now $e\text{-}\stab(\gamma,c)=\stab^\mathrm{I}(\gamma,c)\bigcup \stab^\mathrm{II}(\gamma,c)$. By the computation in the proof of Theorem \ref{stab}, we can conclude that $\stab^\mathrm{I}(\gamma,c)=\stab^\mathrm{I}(\gamma,c)^k$ for all $k\in \ZZ$. When $c=0$, $\stab^\mathrm{II}(\gamma,c)=\emptyset$. Hence, for $(\gamma,0)$, we have $e\text{-}\stab(\gamma,0)=e\text{-}\stab(\gamma,0)^k$ for all $k\in \ZZ$. Also, note that $e\text{-}\stab(\gamma,c)\subseteq e\text{-}\stab(\gamma,c)^k$ for all $k\in \ZZ$. Let $\gamma=\frac{l}{n^p}$ where $l,p\in \ZZ$ with $p>0$ and $\mu=\frac{n^c-1}{n-1}$. Now, suppose $\varphi\in \stab^\mathrm{II}(\gamma,c)^k$. We know from Lemma \ref{basic} that $(\gamma,c)^k=\left(\frac{n^{ck}-1}{n^c-1}\gamma,ck\right)$. Then by  Theorem \ref{e-stab}, we have $\mu_k=\frac{n^{ck}-1}{n-1}$ divides $l_k=\frac{n^{ck}-1}{n^c-1}l$, that is $\frac{l_k}{\mu_k}\in \ZZ$ and $\stab^\mathrm{II}(\gamma,c)^k=\left\{\varphi_{\left(\frac{l_k}{\mu_k},1\right)}\right\}$. But, $\frac{l_k}{\mu_k}=\frac{l}{\mu}\in \ZZ$ and hence $\stab^\mathrm{II}(\gamma,c)=\left\{\varphi_{\left(\frac{l}{\mu},1\right)}\right\}$. But, $\varphi_{\left(\frac{l_k}{\mu_k},1\right)}=\varphi_{\left(\frac{l}{\mu},1\right)}$ and therefore $\stab^\mathrm{II}(\gamma,c)=\stab^\mathrm{II}(\gamma,c)^k$ for all $k>0$. Hence, the theorem.
\end{proof}

\section*{Acknowledgements}
\noindent The authors would like to thank Prof. Amritanshu Prasad, Prof. Parameswaran Sankaran, Prof. Enric Ventura for their valuable comments and suggestions. The first author also thanks IMSc for its hospitality during her visits, when the major part of this collaboration was carried out; she is grateful to the INSPIRE Faculty Fellowship of the
Department of Science and Technology (DST), Government of India.

\end{document}